\documentclass[12pt, reqno]{amsart}
\usepackage{amsmath, amstext, amsbsy, amssymb, amscd}

\newtheorem{theorem}{Theorem}[section]
\newtheorem{lemma}[theorem]{Lemma}

\newtheorem{proposition}[theorem]{Proposition}

\theoremstyle{definition}

\newtheorem{example}[theorem]{Example}

\theoremstyle{remark}
\newtheorem{remark}[theorem]{Remark}
\newtheorem{conjecture}[theorem]{Conjecture}

\numberwithin{equation}{section}

\newcommand{\C}{ \mathbb C }
\newcommand{\U}{\mathfrak{u}}
\newcommand{\hf}{\frac12}
\newcommand{\la}{\lambda}

\newcommand{\spo}{\mathfrak{spo}}

\newcommand{\vdim}{{\rm vdim}}

\newcommand{\Z}{ \mathbb Z }
\newcommand{\gl}{\mathfrak{gl}}
\newcommand{\G}{\mathfrak{g}}
\newcommand{\p}{\mathfrak{p}}
\newcommand{\lev}{\mathfrak{l}}
\newcommand{\h}{\mathfrak{h}}

\begin{document}
\title[ortho-symplectic Lie superalgebras] {Remarks on modules of the ortho-symplectic Lie superalgebras}

\author[Shun-Jen Cheng]{Shun-Jen Cheng}
\address{Institute of Mathematics, Academia Sinica, Taipei, Taiwan
11529}\email{chengsj@math.sinica.edu.tw}
\author[Weiqiang Wang]{Weiqiang Wang}
\address{Department of Mathematics, University of Virginia,
Charlottesville, VA 22904} \email{ww9c@virginia.edu}

\begin{abstract}
We examine in detail the Jacobi-Trudi characters over the
ortho-symplectic Lie superalgebras $\spo(2|2m+1)$ and $\spo(2n|3)$.
We furthermore relate them to Serganova's notion of Euler
characters.
\end{abstract}



 \maketitle
\date{}

\section{Introduction}

The representation theory of Lie superalgebras over $\C$ and that of
modular representations of algebraic groups share similarities in
their lack of complete reducibility and difficulties of finding
irreducible characters. There has been significant progress in
representation theory of Lie superalgebras over $\C$ in recent
years, thanks largely to the works of Brundan, Serganova, and
others, mostly for the so-called type I classical Lie superalgebras
in the classification of Kac \cite{K1} and the queer Lie
superalgebras.

However, the representation theory of the Lie superalgebras
$\spo(2n|\ell)$ with $n\ge 1$ and $\ell \ge 3$ (which are of type II
classical in the sense of \cite{K1}) turns out to be much more
challenging. In her fundamental paper \cite{Ser}, Serganova
announced in 1998 an algorithm of finding the irreducible characters
of $\spo(2n|\ell)$. In the simplest yet already rather nontrivial
case of $\spo(2|3)$, the irreducible characters have been calculated
by Van der Jeugt \cite{VdJ} and Germoni \cite{Ger} using different
approaches (also see Gruson \cite{Gru} and Santos \cite{San} for
related developments).

This Note arises from our attempt to understand Serganova's work
\cite{Ser}, by examining mainly various aspects of the
representation theory of $\spo(2n|3)$ and $\spo(2|2m+1)$. The notion
of Euler characters, which are alternating sums of certain sheaf
cohomology groups, plays a key role in Serganova's theory. In
Section~\ref{sec:Euler}, we introduce the Jacobi-Trudi characters,
imitating the determinant character formula for classical Lie
algebras of type $B, C, D$ (cf.~e.g.~\cite{BB, BLR, CK, KT}). In
general, the Euler characters and the Jacobi-Trudi characters are
only virtual $\spo$-characters. We show that while being different
in general, the Euler characters coincide with the Jacobi-Trudi
characters for a large class of highest weights for $\spo(2n|3)$,
and hence demystify somewhat the notion of Euler characters. Also
some of our Euler character calculations look surprising in light of
\cite{Ser}.

In Section~\ref{sec:tensor}, we determine in a number of cases
when the kernel of a Laplacian on a exterior tensor power of the
natural module is irreducible or reducible. By computations, we
also establish some explicit tensor product decomposition formula
for $\spo(2|2m+1)$-modules.

In the case of $\spo(2|3)$, we provide and compare the explicit
formulas for the ``composition factors" and the virtual dimensions
of the Euler characters, Kac characters, and the Jacobi-Trudi
characters, respectively. We also write down the composition
factors of the tensor products of any irreducible
$\spo(2|3)$-module with the natural module. We end the Note with a
conjecture on a relation betwen the Euler characters of
$\spo(2n|2m+1)$ with respect to the parabolic sublagebra whose
Levi subalgebra is $\gl(n|m)$ and the Kac characters of
$\gl(n|m)$. All these are treated in Section~\ref{sec:examples}.

{\bf Acknowledgment.} S-J.C.~is partially supported by an NSC grant
of the R.O.C.~and an Academia Sinica Investigator grant. The
research of W.W. is partially supported by the NSA and NSF grants.
We thank Ruibin Zhang for many helpful discussions. We also
gratefully acknowledge the support and hospitality of the
NCTS/Taipei Office, the National Taiwan University, the University
of Sydney and the University of Virginia.

\section{The Euler and Jacobi-Trudi characters}
\label{sec:Euler}
\subsection{Preliminaries}

Throughout this paper we shall denote by $\G =\G_{\bar{0}} \oplus
\G_{\bar{1}}$ the ortho-symplectic Lie superalgebra
$\spo(2n|\ell)$ for $n \ge 1$ and $\ell \ge 3$, whose respective
standard Dynkin diagrams and simple roots are as follows (where we
write $\ell =2m+1$ or $2m$):

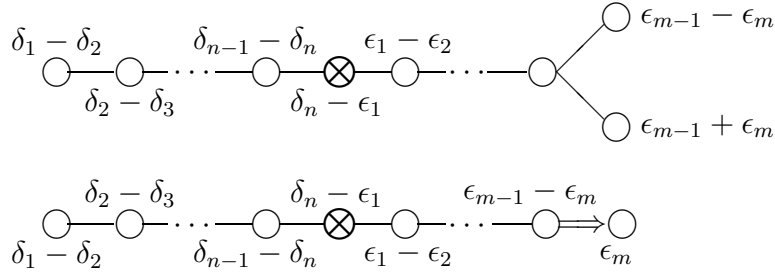
\setlength{\unitlength}{0.16in}
\begin{picture}(24,5)
\put(5,2){\makebox(0,0)[c]{$\bigcirc$}}
\put(7.4,2){\makebox(0,0)[c]{$\bigcirc$}}
\put(11.85,2){\makebox(0,0)[c]{$\bigcirc$}}
\put(14.25,2){\makebox(0,0)[c]{$\bigotimes$}}
\put(16.4,2){\makebox(0,0)[c]{$\bigcirc$}}
\put(20.9,2){\makebox(0,0)[c]{$\bigcirc$}}
\put(23.3,3.8){\makebox(0,0)[c]{$\bigcirc$}}
\put(23.3,0.2){\makebox(0,0)[c]{$\bigcirc$}}
\put(5.35,2){\line(1,0){1.5}} \put(7.82,2){\line(1,0){0.8}}
\put(10.2,2){\line(1,0){1.2}} \put(12.28,2){\line(1,0){1.45}}
\put(14.7,2){\line(1,0){1.25}} \put(16.81,2){\line(1,0){0.9}}
\put(19.1,2){\line(1,0){1.28}} \put(21.35,2){\line(1,1){1.55}}
\put(21.35,2){\line(1,-1){1.55}}
\put(9.5,1.95){\makebox(0,0)[c]{$\cdots$}}
\put(18.5,1.95){\makebox(0,0)[c]{$\cdots$}}
\put(5,3){\makebox(0,0)[c]{$\delta_1-\delta_2$}}
\put(7.4,1){\makebox(0,0)[c]{$\delta_2-\delta_3$}}
\put(11.5,3){\makebox(0,0)[c]{$\delta_{n-1}-\delta_n$}}
\put(14.15,1){\makebox(0,0)[c]{$\delta_n-\epsilon_1$}}
\put(16.5,3){\makebox(0,0)[c]{$\epsilon_1-\epsilon_2$}}
\put(26.3,3.8){\makebox(0,0)[c]{$\epsilon_{m-1}-\epsilon_m$}}
\put(26.3,0.2){\makebox(0,0)[c]{$\epsilon_{m-1}+\epsilon_m$}}
\end{picture}
\bigskip

\setlength{\unitlength}{0.16in}
\begin{picture}(24,4)
\put(5,2){\makebox(0,0)[c]{$\bigcirc$}}
\put(7.4,2){\makebox(0,0)[c]{$\bigcirc$}}
\put(11.85,2){\makebox(0,0)[c]{$\bigcirc$}}
\put(14.25,2){\makebox(0,0)[c]{$\bigotimes$}}
\put(16.4,2){\makebox(0,0)[c]{$\bigcirc$}}
\put(21,2){\makebox(0,0)[c]{$\bigcirc$}}
\put(23.5,2){\makebox(0,0)[c]{$\bigcirc$}}
\put(5.35,2){\line(1,0){1.5}} \put(7.82,2){\line(1,0){0.8}}
\put(10.2,2){\line(1,0){1.2}} \put(12.28,2){\line(1,0){1.45}}
\put(14.7,2){\line(1,0){1.25}} \put(16.81,2){\line(1,0){0.9}}
\put(19.25,2){\line(1,0){1.28}}
\put(21.35,1.75){$\Longrightarrow$}
\put(9.5,1.95){\makebox(0,0)[c]{$\cdots$}}
\put(18.5,1.95){\makebox(0,0)[c]{$\cdots$}}
\put(5,1){\makebox(0,0)[c]{$\delta_1-\delta_2$}}
\put(7.4,3){\makebox(0,0)[c]{$\delta_2-\delta_3$}}
\put(11.5,1){\makebox(0,0)[c]{$\delta_{n-1}-\delta_n$}}
\put(14.15,3){\makebox(0,0)[c]{$\delta_n-\epsilon_1$}}
\put(16.5,1){\makebox(0,0)[c]{$\epsilon_1-\epsilon_2$}}
\put(20.5,3){\makebox(0,0)[c]{$\epsilon_{m-1}-\epsilon_m$}}
\put(23.3,1){\makebox(0,0)[c]{$\epsilon_{m}$}}
\end{picture}
\medskip

\noindent Here and further we use the standard notation for the
simple roots of the ortho-symplectic Lie superalgebra, i.e.
$\{\epsilon_1,\ldots,\epsilon_m,\delta_1,\ldots,\delta_n\}$ denote
the standard dual basis of the standard Cartan subalgebra $\h$
equipped with a symmetric bilinear form determined by
$(\epsilon_i,\epsilon_j)=-\delta_{ij}$, $(\delta_i,\epsilon_j)=0$,
and $(\delta_i,\delta_j)=\delta_{ij}$. Recall that an odd root
$\alpha$ is called {\em isotropic} if $(\alpha,\alpha)=0$. Note that
$\spo(2n|2m)_{\bar{0}} =\mathfrak{sp}(2n) \oplus \mathfrak{so}(2m)$
and $\spo(2n|2m+1)_{\bar{0}} =\mathfrak{sp}(2n) \oplus
\mathfrak{so}(2m+1)$.

In concrete matrix form, the Lie superalgebra $\spo(2n|2m+1)$
consists of the $(2n+2m+1) \times (2n+2m+1)$ matrices in the
following $(n|n|m|m|1)$-block form
\begin{eqnarray} \label{matrixSPO}
g = \left[
    \begin{array}{rrrrr}
    d      &   e     & y_1^t  & x_1^t &  z_1^t       \\
    f      &  -d^t   & -y^t   & - x^t &  -z^t  \\
    x      & x_1     &  a     &  b    &  -v^t  \\
    y      & y_1     &  c     & -a^t  &  -u^t   \\
    z      & z_1    & u      &  v    &    0
    \end{array}
    \right],
\end{eqnarray}
where  $b, c$  are skew-symmetric, and $e, f$ are symmetric
matrices. The remaining $a,d,x,y,x_1,y_1, z, z_1, u, v$ are
arbitrary matrices of respective sizes. Similarly, the Lie
superalgebra $\spo(2n|2m)$ consists of the $(2n+2m) \times
(2n+2m)$ matrices that are obtained from $g$ of the form
(\ref{matrixSPO}) with the last row and column deleted. The
natural $\G$-module will be denoted by $V$.

Let $\la=\sum_{i=1}^n{a_i}\delta_i+\sum_{j=1}^m
b_j\epsilon_j\in\h^*$ with $a_i,b_j\in\Z$.  Denote by $L(\la)$ the
highest weight irreducible $\G$-module of highest weight $\la$. The
following goes back to \cite{K1}.

\begin{proposition}\label{finite-dim}
Let $\la=\sum_{i=1}^n{a_i}\delta_i+\sum_{j=1}^m b_j\epsilon_j$ be as
above.
\begin{itemize}
\item[(i)] Suppose that $\G=\spo(2n|2m)$.  Then $L(\la)$ is
finite-dimensional if and only if $a_1\ge \cdots\ge a_n\ge 0$,
$b_1\ge \cdots b_{m-1}\ge|b_m|$, and $a_n<m$ implies that
$b_{a_n+1}=\cdots=b_m=0$.
\item[(ii)] Suppose that $\G=\spo(2n|2m+1)$.  Then $L(\la)$ is
finite-dimensional if and only if $a_1\ge \cdots\ge a_n \ge 0$,
$b_1\ge \cdots b_{m-1}\ge b_m\ge 0$, and $a_n<m$ implies that
$b_{a_n+1}=\cdots=b_m=0$.
\end{itemize}
\end{proposition}

A weight $\la\in\h^*$ satisfying the condition in Proposition
\ref{finite-dim} will be called a {\em dominant weight} for $\G$.

\begin{remark}\label{rem:hook}
It can be seen that the set of dominant weights for
$\spo(2n|2m+1)$ is in 1-1 correspondence with the set of
partitions $\la$ with $\la_{n+1}\le m$ as follows.  For such $\la$
let
$\la^\sharp=(\la_1,\ldots,\la_n,\langle\la'_1-n\rangle,\ldots,\langle\la'_m-n\rangle)$,
where $\langle \ell\rangle=\ell$, if $\ell\ge 0$, and zero
otherwise. We may regard $\la^\sharp$ as the weight
$\sum_{i=1}^n\la^\sharp_i\delta_i+\sum_{j=1}^m\la^\sharp_j\epsilon_j$.
The map $\la\rightarrow\la^\sharp$ is a bijection. Similarly the
set of such partitions parameterizes the set of
$\spo(2n|2m)$-dominant weights with $b_m\ge 0$.
\end{remark}

The graded half sum of positive roots $\rho$ is given by
\begin{align*}
\rho&=\sum_{i=1}^n
(i-m)\delta_{n-i+1}+\sum_{j=1}^m(m-j)\epsilon_j,\quad \text{if } \G=\spo(2n|2m),\\
\rho&=\sum_{i=1}^n
(i-m-\hf)\delta_{n-i+1}+\sum_{j=1}^m(m-j+\hf)\epsilon_j,\quad
\text{if } \G=\spo(2n|2m+1).
\end{align*}

Let $\la\in\h^*$. The action of the center $Z(\G)$ of the
universal enveloping algebra $U(\G)$ on $L(\la)$ defines an
algebra homomorphism (called the {\em central character})
$\chi_\la:Z(\G)\rightarrow\C$. The following was first stated in
\cite{K2}, and proved in \cite{Sv2} (also cf. \cite{Go}).

\begin{proposition}\label{central:char}
Let $\la,\mu\in\h^*$ and let $W$ be the Weyl group of $\G$. We have
$\chi_\la=\chi_\mu$ if and only if there exists a sequence of
isotropic roots $\alpha_1,\ldots,\alpha_l$ and $w\in W$ such that
$\mu=w(\la+\rho+\alpha_1+\cdots+\alpha_l)-\rho$, and
$(\la+\rho+\alpha_1+\cdots+\alpha_{s-1},\alpha_s)=0$, for all
$s=1,\ldots,l$.
\end{proposition}

Recall \cite{K1} that $\la$ is called {\em typical} if
$(\la+\rho,\alpha)\not=0$, for all isotropic root $\alpha$.
Otherwise $\la$ is called {\em atypical}.

\subsection{The Euler characters}

Let $\p$ be a parabolic subalgebra of $\G$ with Levi subalgebra
$\lev$ and nilradical $\U$. Write
$\G=\lev\oplus\U\oplus\tilde{\U}$. Let $M$ be a finite-dimensional
irreducible $\lev$-module extended trivially to an irreducible
$\p$-module. A key role in Serganova's algorithm of finding the
irreducible finite-dimensional characters of $\G$ was played by
the notion of Euler characters \cite[(1.2)]{Ser}. By definition,
the Euler characters are alternating sums of certain sheaf
cohomology groups that themselves are finite-dimensional
$\G$-modules, and hence are in general not $\G$-characters but
only virtual $\G$-characters. According to \cite{Ser}, a formula
for the Euler character is given by
\begin{equation*}
E^{\mathfrak{p}}(M)=D\sum_{w\in W}(-1)^{l(w)}w\left(\frac{e^\rho{\rm
ch}M}{\prod_{\alpha\in\Delta^1_{{\mathfrak
l},+}}(1+e^{-\alpha})}\right),
\end{equation*}
where ${\rm ch} M$ denotes the character of an $\lev$-module $M$,
and 
$$
D= {D_1}/{D_0},
 \quad
  D_1=\prod_{\alpha\in\Delta^1_+}
(e^{\frac{\alpha}{2}}+e^{\frac{-\alpha}{2}}),
 \quad
D_0=\prod_{\alpha\in\Delta^0_+}
(e^{\frac{\alpha}{2}}-e^{\frac{-\alpha}{2}}).
$$
For our purpose, we take this as the definition of some
distinguished virtual $\G$-characters. Furthermore
$\Delta^1_{\mathfrak l,+}$ (respectively $\Delta^0_{\mathfrak l,+}$)
denotes the set of positive odd (respectively even) roots in
$\mathfrak{l}$. The simple $\mathfrak{l}$-module of highest weight
$\la$ will be denoted by $L^0(\la)$.

Let $\mathfrak{b}$ be the standard Borel subalgebra of $\G$ and let
$\mathfrak{h}$ be the corresponding Cartan subalgebra. For a
finite-dimensional highest weight $\la\in\h^*$ let $\C_\la$ be the
one-dimensional $\mathfrak{b}$-module on which $\mathfrak{h}$
transforms by $\la$. We call $K(\la) :=E^{\mathfrak b}(\C_\la)$ the
{\em Kac (virtual) character} (of ``highest weight $\la$"). When
$\la$ is a typical dominant weight, indeed $K(\la)$ is the character
of the simple $\G$-module $L(\la)$ \cite{K3}.

\begin{remark}\label{rem:euler} In \cite[(12)]{San} Santos defined the functor
$\mathcal L_0$ that may be regarded as a super analogue of the
Bernstein-Zuckerman functor (see e.g.~\cite{KV}).  If we denote
the $i$th derived functor by $\mathcal L_i$, then it can be shown
that
$$\sum_{i\ge 0}(-1)^i{\rm ch}\mathcal L_i({\rm Ind}_\p^\G
M)=E^\p(M).$$ Since all the irreducible composition factors of the
finite-dimensional $\G$-module $\mathcal L_i({\rm Ind}_\p^\G M)$
have the same central character by \cite[Proposition 4.5]{San}, it
follows that all irreducible composition factors of $E^\p(M)$ have
the same central character.
\end{remark}

\begin{example}\label{ex:incon}
Consider $\mathfrak{spo}(2|4)$ with the maximal parabolic subalgebra
$\mathfrak{p}$ obtained by removing the simple root
$\epsilon_1+\epsilon_2$.  The Levi subalgebra $\lev$ is isomorphic
to $\mathfrak{l}=\mathfrak{gl}(1|2)$. Let $\C$ be the trivial
module, and let $\C^{1|2}$ be the natural $\gl(1|2)$-module of
highest weight $\delta_1$. One can show that
$$
E^{\mathfrak{p}}(\C) =2\left[{\rm ch} \left(\C\right)\right],
 \quad E^{\mathfrak{p}}(\C^{1|2})={\rm ch}\left(\C^{2|4}\right),\quad
E^{\mathfrak{p}}(L^0(\delta_1+\epsilon_1))={\rm ch}
L(\delta_1+\epsilon_1).
$$
The factor of $2$ in the formula of the first example is
inconsistent with \cite[Theorem 3.3]{Ser}, which in this case
predicts the irreducibility of the Euler character with
non-vanishing cohomology appearing only in degree zero. Also
\cite[Proposition 3.4]{Ser} in this case further seems to imply that
$\C$ should only appear with multiplicity one. On the other hand the
last two examples are consistent with \cite{Ser} \footnote{In a
private e-mail communication Serganova has informed us that she was
aware of these inconsistencies. We were further told that she has
found methods to amend these problems.  We thank her for kindly
sharing this information with us.}.
\end{example}

\begin{example}
Consider $\mathfrak{spo}(2|6)$ with the maximal parabolic subalgebra
$\mathfrak{p}$ obtained by removing the simple root
$\epsilon_2+\epsilon_3$.  The Levi subalgebra $\lev$ is isomorphic
to $\mathfrak{l}=\mathfrak{gl}(1|3)$. Let $\C$ be the trivial module
and let $\C^{1|3}$ be the standard $\gl(1|3)$-module of highest
weight $\delta_1$. One has
$$
E^{\mathfrak{p}}(\C)=2\left[{\rm
ch}\left(\C\right)\right],\quad
E^{\mathfrak{p}}(\C^{1|3})
=2\left[{\rm ch} \left(\C^{2|6}\right)\right],
 \quad
E^{\mathfrak{p}}((S^2(\C^{1|3})))={\rm ch} S^2(\C^{2|6}).
$$
\end{example}

\subsection{The Jacobi-Trudi characters}

Let $\la$ be a partition with $\la_{n+1}\le m$ and let $k$ be the
length of $\la$. We identify $\la^\sharp$ with
$\sum_{i=1}^n\la^\sharp_i\delta_i+\sum_{j=1}^m\la^\sharp_j\epsilon_j$.
The {\em Jacobi-Trudi character} $D(\la^\sharp)$ is defined to be
the determinant of the following matrix
\[\begin{pmatrix}
p_{\la^*}, p_{\la^*+(1^k)}+p_{\la^*-(1^k)},\ldots,
p_{\la^*+(k-1)(1^k)}+p_{\la^*-(k-1)(1^k)}
\end{pmatrix}.\]
Here for a partition $\mu=(\mu_1,\ldots,\mu_k)$ we let
$\mu^*=(\mu_1,\mu_2-1,\ldots,\mu_k-k+1)$, and for a $k$-tuple of
integers $a=(a_1,\ldots,a_k)$, $p_a$ stands for the column vector
$(p_{a_1}, \ldots, p_{a_k})^t$, where $p_{a_j}$ is the character
of $S^{a_j}(V)$. In general the Jacobi-Trudi character is not a
$\G$-character, but only a virtual $\G$-character. See \cite{BLR}
for some related discussions.

\begin{lemma}  \label{lem:p=e}
Let $\la$ be a partition and suppose that $\mu=\la'$ has length
$\ell$. Then $D(\la^\sharp)$ is equal to the determinant of the
matrix
\[\begin{pmatrix}
e_{\mu^*}-e_{\mu^*-2(1^\ell)},
e_{\mu^*+(1^\ell)}-e_{\mu^*-3(1^\ell)},\ldots, e_{\mu^*+(\ell
-1)(1^\ell)}-e_{\mu^*-(\ell+1)(1^\ell)}
\end{pmatrix},\]
where for an $\ell$-tuple of integers $a=(a_1,\ldots,a_\ell)$,
$e_a$ stands for the column vector $(e_{a_1}, \ldots,
e_{a_\ell})^t$, and $e_{a_j}$ is the character of
$\Lambda^{a_j}(V)$.
\end{lemma}

\begin{proof} Denote by $|a_{ik}|$ the determinant of a square matrix
$[a_{ik}]$. Recall the classical identity (see
e.g.~\cite[Proposition 2.3.3]{KT})
\begin{align*}
&\left|p_{\la^*}, p_{\la^*+(1^k)}+p_{\la^*-(1^k)},\ldots,
p_{\la^*+(k-1)(1^k)}+p_{\la^*-(k-1)(1^k)}\right|=\\
&\left| e_{\mu^*}-e_{\mu^*-2(1^\ell)},
e_{\mu^*+(1^\ell)}-e_{\mu^*-3(1^\ell)},\ldots, e_{\mu^*+(\ell
-1)(1^\ell)}-e_{\mu^*-(\ell+1)(1^\ell)} \right|,
\end{align*}
where $p_t$ and $e_t$, $t\in\Z_+$, are the complete symmetric and
elementary symmetric functions in the variables $y_1,y_2,\cdots$,
respectively. We regard this identity as a symmetric function
identity in the variables $y_{2n+1},y_{2n+2},\cdots$ and apply the
involution that interchanges the complete and the elementary
symmetric functions (see e.g.~\cite{Mac}).  Now in the resulting
identity we set $y_{2n+2m+1}=y_{2n+2m+2}=\cdots=0$ in the case of
$\spo(2n|2m)$, and
$1-y_{2n+2m+1}=y_{2n+2m+2}=y_{2n+2m+3}=\cdots=0$ in the case of
$\spo(2n|2m+1)$. Next we set $y^{-1}_{j}=y_{n+j}$, for
$j=1,2,\cdots,n$, and $y^{-1}_{2n+j}=y_{2n+m+j}$, $j=1,\ldots,m$.
Finally, setting $e^{\delta_i}=y_i$, and
$e^{\epsilon_j}=y_{2m+j}$, $i=1,\ldots,n$, and $j=1,\ldots,m$ we
obtain the desired identity.
\end{proof}

\subsection{Euler versus Jacobi-Trudi characters}

Define
\begin{align*}
&\varphi_0(z)=\prod_{k=1}^n (1-u_kz)(1-u_k^{-1}z),\\
&\varphi_1(z)=(1+yz)(1+y^{-1}z)(1+z).
\end{align*}
Note we have the following interpretation
\begin{equation} \label{sympower}
\frac{\varphi_1(z)}{\varphi_0(z)} =\sum_l {\rm ch} (S^l(\C^{2n|3}))
z^l.
\end{equation}

\begin{lemma}
The following identities hold:
\begin{align}
 \prod_{i} & \varphi_0(z_i) \cdot
\left|\frac{1}{(1-z_iu_k)(1-z_iu_k^{-1})}\right| \label{cauchy1}\\
& ={\prod_i {z_i^{n-1}}
\left|u^{n-1}+u^{-n+1},u^{n-2}+u^{-n+2},\ldots,1\right|\;
\left|1,z^{-1}+z,,\ldots,z^{n-1}+z^{-n+1}\right|}, \nonumber \\
%
\left|u^{n} \right. & \left.-u^{-n},u^{n-1}-u^{-n+1},\ldots,u-u^{-1}\right| \label{cauchy3}\\
& =\prod_i(u_i-u_i^{-1})
\left|u^{n-1}+u^{-n+1},u^{n-2}+u^{-n+2},\ldots,1\right|, \nonumber
\\
\left| \right. & z^{-1}  \left. -z,\ldots,z^{-n}-z^n\right|  \label{cauchy9} \\
&=(z_n^{-1} -z_n)\prod_{i=1}^{n-1}z_i^{-1}(1 -z_iz_n)(1
-z_iz_n^{-1})\left|z-z^{-1},\ldots,z^{n-1}-z^{-n+1}\right|,
\nonumber
\\
\sum_{l\ge 0} & (u^{l+\hf}-u^{-l-\hf})z^{l-1}
=(1+z^{-1})(u^\hf-u^{-\hf})\frac{1}{(1-uz)(1-u^{-1}z)}.
\label{cauchy4}
\end{align}
\end{lemma}

\begin{proof}
Recall Cauchy's formula (cf. \cite{Wey})
\begin{equation*}
\prod_{i,k}{(x_i-y_k)} \cdot
\left|\frac{1}{x_i-y_k}\right|=\left|1,x,\ldots,x^{n-1}\right|\cdot
\left|1,y,\ldots,y^{n-1}\right|,
\end{equation*}
where $x^j$ $(0\le j \le n-1)$ denotes the column vector
$(x^j_1,\ldots,x^j_n)^t$ and so on; similar notations apply below.
Putting $x_i=z_i-z_i^{-1}$ and $y_k=u_k-u_k^{-1}$ in the Cauchy's
formula gives us (\ref{cauchy1}). The identity (\ref{cauchy3}) is
obtained by taking out the common factors $(u_i-u_i^{-1})$ of the
determinant on its left-hand side, and applying elementary column
operations. The identity (\ref{cauchy9}) is proved by using twice
the Weyl denominator formula for $\mathfrak{sp}(2n)$:
\begin{align*}
&\left|z^{-1}-z,\ldots,z^{-n}-z^n\right| = (-1)^n
\left|z-z^{-1},\ldots,z^{n}-z^{-n}\right|\nonumber\\
&=(-1)^n\prod_{1\le i<j\le n}(z_j+z_j^{-1}-z_i-z_i^{-1})
\prod_{1\le i\le n}(z_i-z_i^{-1})\nonumber\\
&=(z_n^{-1} -z_n) \prod_{i=1}^{n-1}(z_i+z_i^{-1}-z_n-z_n^{-1})\left|z-z^{-1},\ldots,z^{n-1}-z^{-n+1}\right|\nonumber\\
&=(z_n^{-1} -z_n)\prod_{i=1}^{n-1}z_i^{-1}(1 -z_iz_n)(1
-z_iz_n^{-1})\left|z-z^{-1},\ldots,z^{n-1}-z^{-n+1}\right|.
\end{align*}
The (\ref{cauchy4}) follows by the geometric series expansion of its
right-hand side.
\end{proof}

\begin{theorem}\label{det=euler}
Let $\G=\spo(2n|3)$ and $\p$ be the parabolic subalgebra obtained by
removing the simple roots
$\delta_1-\delta_2,\ldots,\delta_{n-1}-\delta_n$ so that
$\lev=\gl(1)^{n-1}\oplus\spo(2|3)$. Let
$\la=\sum_{i=1}^{n-1}\la_i\delta_i$ be dominant integral.  Then the
Euler characters coincide with the Jacobi-Trudi characters, that is,
$E^\p(L^0(\la))=D(\la).$
\end{theorem}

\begin{proof}
Denote by $D_0^{sp}$ the Weyl denominator for $\mathfrak{sp}(2n)$.
We have
\begin{align*}
 &E^\p  (L^0(\la)) \\
 & = \frac{D_1}{D_0}\sum_{w\in W}(-1)^{l(w)}w
 \left(
\frac{e^{\la+\rho}}{(1+e^{-\delta_n+\epsilon_1})(1+e^{-\delta_n-\epsilon_1})(1+e^{-\delta_n})}
\right)\\
&=\frac{D_1}{D_0^{sp}(e^{\epsilon_1/2}-e^{-\epsilon_1/2})}
\sum_{w\in W_{sp}}(-1)^{l(w)} w\left(
\frac{e^{{\sum_{i=1}^n}(\la_i+n-i-\hf)\delta_i}(e^{\epsilon_1/2}-e^{-\epsilon_1/2})}
{(1+e^{-\delta_n+\epsilon_1})(1+e^{-\delta_n-\epsilon_1})(1+e^{-\delta_n})}
\right)\\
&=\frac{D_1}{D_0^{sp}} \sum_{s\in S_n}(-1)^{l(s)}
s\left(\sum_{w\in\Z_2^n}w\left(
\frac{e^{{\sum_{i=1}^{n-1}}(\la_i+n-i-\hf)\delta_i-\hf\delta_n}}
{(1+e^{-\delta_n+\epsilon_1})(1+e^{-\delta_n-\epsilon_1})(1+e^{-\delta_n})}
\right)\right)\\
&=\frac{D_1}{D_0^{sp}} \sum_{s\in S_n}(-1)^{l(s)} s\left(
\prod_{i=1}^{n-1}(e^{(\la_i+n-i-\hf)\delta_i}-e^{-(\la_i+n-i-\hf)\delta_i})
\frac{e^\hf\delta_n-e^{-\hf\delta_n}}{e^\delta_n+e^{-\delta_n}+e^{\epsilon_1}+e^{-\epsilon_1}}
\right).
\end{align*}
Setting $\la_i+n-i-\hf=\mu_i$, for $i=1,\ldots,n-1$,
$u_i=e^{\delta_i}$, and $y=e^{\epsilon_1}$, we get
\begin{equation}\label{cauchy5}
E^\p(L^0(\la))=\frac{D_1}{D_0^{sp}}\left|
u^{\mu_1}-u^{-\mu_1},\ldots,u^{\mu_{n-1}}-u^{-\mu_{n-1}},\frac{u^\hf-u^{-\hf}}{u+u^{-1}+y+y^{-1}}
\right|.
\end{equation}

We write $D_1=D'_1 \cdot {\prod_k (u_k^\hf+u_k^{-\hf})}$, where
\begin{align}\label{cauchy6}
D_1' =y^{-n} \varphi_0(-y) 
= \prod_{k}(u_k^\hf y^\hf+u_k^{-\hf}y^{-\hf})(u_k^\hf
y^{-\hf}+u_k^{-\hf}y^\hf).
\end{align}
It is convenient to set $z_n=-y$ and denote by $[z_1^{a_1}\ldots
z_k^{a_k}]f$ the coefficient of $z_1^{a_1}\cdots z_k^{a_k}$ of $f$
below. Noting that
$$\frac{1}{u_k+u_k^{-1}+y+y^{-1}}=\frac{y}{(1-u_kz_n)(1-u_k^{-1}z_n)}$$
and using (\ref{cauchy4}), we rewrite (\ref{cauchy5}) as

\begin{align*}
E^\p(L^0(\la))&=[z_1^{\mu_1-\frac{3}{2}}\cdots
z^{\mu_{n-1}-\frac{3}{2}}_{n-1}]\frac{D_1}{D_0^{sp}} \left|
\frac{(1+z_i^{-1})(u_k^\hf-u_k^{-\hf})}{(1-u_kz_i)(1-u_k^{-1}z_i)};\frac{y(u_k^\hf-u_k^{-\hf})}
{(1-u_kz_n)(1-u_k^{-1}z_n)} \right|\\
&=[z_1^{\mu_1-\frac{3}{2}}\cdots
z^{\mu_{n-1}-\frac{3}{2}}_{n-1}]\frac{D'_1\prod
(u_k^\hf+u_k^{-\hf})}{D_0^{sp}}\times\\
&\qquad y \prod(u_k^\hf-u_k^{-\hf})\prod(1+z_i)^{-1}\left|
\frac{1}{(1-u_kz_i)(1-u_k^{-1}z_i)} \right|\\
&=[z_1^{\mu_1-\frac{3}{2}}\cdots
z^{\mu_{n-1}-\frac{3}{2}}_{n-1}]\frac{D'_1}{D_0^{sp}}
\prod(u_k-u_k^{-1})\frac{y\prod(z_i^{-1}-z_i)}{(y-y^{-1})\prod(1-z_i)}\times\\
&\qquad\frac{\prod_{i=1}^{n}z_i^{n-1}}{\prod_{i=1}^n\varphi_0(z_i)}\left|u^{n-1}+u^{-n+1},\ldots,1\right|
\left|1,z^{-1}+z,\ldots,z^{n-1}+z^{-n+1}\right| \\
& =[\prod_{i=1}^{n-1} z_i^{\la_i-i-1}]\frac{D'_1}{D_0^{sp}}
\frac{y^n(-1)^{n-1}}{\prod_{i=1}^{n-1}(1-z_i)(y-y^{-1})\varphi_0(-y)\prod_{i=1}^{n-1}\varphi_0(z_i)}\\
&\qquad\left|u^{n}-u^{-n},\ldots,u-u^{-1}\right|
\left|z^{-1}-z,\ldots,z^{-n}-z^{n}\right|,
\end{align*}
where (\ref{cauchy3}) was used in the last equation. By
(\ref{cauchy9}), (\ref{cauchy6}) and the Weyl denominator formula
$D_0^{sp}=\left|u^n-u^{-n},\ldots,u-u^{-1}\right|,$
 we rewrite the above expression for $E^\p(L^0(\la))$ as
\begin{align*}
 &= [\prod_{i=1}^{n-1} z_i^{\la_i-i-1}]
\frac{(-1)^{n-1}}{(y-y^{-1})\prod_{i=1}^{n-1}\varphi(z_i)(1-z_i)}
\left|z^{-1}-z,\ldots,z^{-n}-z^{n}\right| \\
&= [\prod_{i=1}^{n-1} z_i^{\la_i-i-1}]
\frac{\prod_{i=1}^{n-1}z_i^{-2}(1+z_i)(1+yz_i)(1+y^{-1}z_i)}{\prod_{i=1}^{n-1}\varphi_0(z_i)(z_i-z_i^{-1})}
\left|z-z^{-1},\ldots,z^{n-1}-z^{-n+1}\right|\\
&= [\prod_{i=1}^{n-1} z_i^{\la_i-i+1}]
\frac{\varphi_1(z_i)}{\varphi_0(z_i)}
\left|z-z^{-1},\ldots,z^{n-1}-z^{-n+1}\right|,
\end{align*}
which coincides with the Jacobi-Trudi character $D(\la)$ by
(\ref{sympower}).
\end{proof}

\begin{remark}
Let $\mathfrak{q}$ be the parabolic with
$\lev=\gl(1)^{n-1}\oplus\gl(1|1)$. Then one can show similarly that
$E^{\mathfrak q}(L^0(\la))=2E^\p(L^0(\la))$.
\end{remark}

\begin{proposition} Assume $n\ge 2$ and
let $\G=\spo(2n|3)$. Let $\p$ be the parabolic subalgebra obtained
by removing the simple roots
$\delta_1-\delta_2,\ldots,\delta_{n-1}-\delta_n$ and let
$\la=\sum_{i=1}^{n-1}\la_i$ be a dominant integral weight of
$\spo(2n|3)$ with $\la_{n-1} >0$.  Set $\tilde{\la}=\la+\delta_n$.
Then we have
\begin{equation*}
E^\p(L^0(\tilde{\la}))=K (\tilde{\la}) +E^\p(L^0(\la)).
\end{equation*}
\end{proposition}
\noindent (Note that $\tilde{\la}$ and $\la$ have the same central
character.)

\begin{proof}
We have
\begin{align*}
E^\p& (L^0(\tilde{\la}))  \\
 &= \frac{D_1}{D_0}\sum_{w\in W}(-1)^{l(w)}
w\left(
\frac{e^{\la+\rho}(e^{\delta_n}+e^{-\delta_n}+1+e^{\epsilon_1}+e^{-\epsilon_1})}
{(1+e^{-\delta_n+\epsilon_1})(1+e^{-\delta_n-\epsilon_1})(1+e^{-\delta_n})}\right)
 \\
 & = \frac{D_1}{D_0^{sp}}\left|
u^{\mu_1}-u^{-\mu_1},\ldots,u^{\mu_{n-1}}-u^{-\mu_{n-1}},(u^\hf-u^{-\hf})+\frac{u^\hf-u^{-\hf}}{u+u^{-1}+y+y^{-1}}
\right|,
\end{align*}
where $\tilde{\la}_i+n-i-\hf=\mu_i$, for $i=1,\ldots,n$,
$u_i=e^{\delta_i}$, and $y=e^{\epsilon_1}$. The last identity
above was derived in a way similar to (\ref{cauchy5}). Thus by
comparing with (\ref{cauchy5}) we have
\begin{align*}
E^\p(L^0(\tilde{\la})) =E^\p(L^0(\la)) + \left|
u^{\mu_1}-u^{-\mu_1},\ldots,u^{\mu_{n-1}}-u^{-\mu_{n-1}},(u^\hf-u^{-\hf})
\right|.
\end{align*}
Now similarly to the derivation of (\ref{cauchy5}), we can again
show that
\begin{align*}
K(\tilde{\la})&=\frac{D_1}{D_0}\sum_{w\in
W}(-1)^{l(w)}w\left(e^{\tilde{\la}+\rho}\right)\\
&=\left|
u^{\mu_1}-u^{-\mu_1},\ldots,u^{\mu_{n-1}}-u^{-\mu_{n-1}},(u^\hf-u^{-\hf})
\right|,
\end{align*}
which completes the proof.
\end{proof}

\section{Tensors of $\G$-modules}
\label{sec:tensor}

\subsection{Exterior tensors of the natural module}

Let $\xi_j,\bar{\xi}_j$, $j=1,\ldots,n$, be the
standard basis for $\C^{2n|0}$, and let $x_i,\bar{x}_i,x_0$,
$i=1,\ldots,m$, be the standard basis for $\C^{0|2m+1}$, so that
$\Lambda(\C^{2n|2m+1})\cong\C[x_i,\bar{x}_i,x_0]\otimes\Lambda(\xi_j,\bar{\xi}_j)$.
We consider the Laplacian for $\spo(2n|2m+1)$
$$\Delta=\sum_{j=1}^n\frac{\partial}{\partial\xi_j}\frac{\partial}{\partial\bar{\xi}_j}-
\sum_{i=1}^m\frac{\partial}{\partial
x_i}\frac{\partial}{\partial\bar{x}_i}-\frac{1}{2}
\frac{\partial^2}{\partial x^2_0}.
$$
Dropping the last term in $x_0$ gives us the Laplacian for
$\spo(2n|2m)$.

\begin{proposition}\label{exterior:tensor}
\begin{itemize}
\item[(i)]
Let $\G=\spo(2|2m+1)$ with $m\ge 1$, and let $k \ge 1$. Then the
kernel of
$\Delta:\Lambda^k(\C^{2|2m+1})\rightarrow\Lambda^{k-2}(\C^{2|2m+1})$
as a $\G$-module is irreducible and isomorphic to
$L_{\delta_1+(k-1)\epsilon_1}$.

\item[(ii)]
Let $\G=\spo(2n|3)$, with $n\ge 1$, and let $1\le k\le n-1$. Then
the kernel of
$\Delta:\Lambda^k(\C^{2n|3})\rightarrow\Lambda^{k-2}(\C^{2n|3})$ as
a $\G$-module is irreducible.

\item[(iii)]
Let $n\ge 2$. The $\spo(2n|2n)$-module ${\rm
ker}\left[\Delta:\Lambda^2(\C^{2n|2n})\rightarrow\C\right]$ is not
irreducible.  It contains a unique submodule which is the trivial
module.

\item[(iv)] Let $n\ge 2$. The $\spo(2n|2n)$-module ${\rm
ker}\left[\Delta:\Lambda^k(\C^{2n|2n})\rightarrow\Lambda^{k-2}(\C^{2n|2n})\right]$
is irreducible, for $k=3,4$.
\end{itemize}
\end{proposition}

\begin{proof}
(i). The case of $k=1$ is clear, so we may assume that $k\ge 2$. We
decompose $\Lambda^k(\C^{2|2m+1})$ as an
${\mathfrak{sp}}(2)\oplus{\mathfrak{so}(2m+1)}$-module. By a direct
computation, as a
${\mathfrak{sp}}(2)\oplus{\mathfrak{so}(2m+1)}$-module, the kernel
of $\Delta$ in $\Lambda^k(\C^{2|2m+1})$ is a direct sum of three
irreducibles with respective highest weight vectors $x_1^k$,
$x_1^{k-1}\xi_1$ and $v :=
x_1^{k-2}\left((k+m-\frac{3}{2})\xi_1\bar{\xi}_1-\sum_{i=1}^m
x_i\bar{x}_i-\hf x_0^2\right)$, which have distinct weights. Among
these three vectors, only the weight of $x_1^{k-1}\xi_1$ can be a
finite-dimensional $\spo(2|2m+1)$-highest weight.
From this it follows that the kernel of $\Delta$ on
$\Lambda^k(\C^{2|2m+1})$ is an irreducible $\spo(2|2m+1)$-module.

(ii). Under our assumption, the kernel of $\Delta$ in
$\Lambda^k(\C^{2n|3})$ is isomorphic to $D(\la)$ with $\la =\delta_1
+\ldots + \delta_k$ by Lemma~\ref{lem:p=e}. Also by
Theorem~\ref{det=euler} and Remark~\ref{rem:euler} we see that all
composition factors of $D(\la)$ lie in the same block. Suppose on
the contrary that $D(\la)$ were not irreducible. As an
$\mathfrak{sp}(2n)\oplus\mathfrak{so}(3)$-module, ${\rm
ker}\left[\Delta:\Lambda^k(\C^{2n|3})\rightarrow\Lambda^{k-2}(\C^{2n|3})\right]$
decomposes into a direct sum of the irreducibles of highest weights
of the form $\sum_{i=1}^{k-j}\delta_i+s\epsilon_1$, among which only
weights of the form $\sum_{i=1}^{k-j}\delta_i$ can possibly be
finite-dimensional $\spo(2n|3)$-highest weights by
Proposition~\ref{finite-dim}. Hence, $D(\la)$ must have an
$\spo(2n|3)$-singular vector with weight $\sum_{i=1}^{k-j}\delta_i$,
for $j>0$.   However, a calculation using Proposition
\ref{central:char} shows that $\sum_{i=1}^{k-j}\delta_i$ have
different central characters for distinct $j$, which is a
contradiction.

(iii). One shows that as an
$\mathfrak{sp}(2n)\oplus\mathfrak{so}(2n)$-module, ${\rm ker}\Delta$
has $x_1^2$, $x_1\xi_1$, $\xi_1\xi_2$, and $\phi :=\sum_{i=1}^n
\xi_i\bar{\xi}_i-\sum_{i=1}^n x_i\bar{x}_i$ as  a complete set of
highest weight vectors. Among the weights of these vectors only
$\delta_1+\delta_2$ and $0$ are finite-dimensional
$\spo(2n|2n)$-weights, which implies that ${\rm ker}\Delta$ has at
most two composition factors. However, by a direct computation,
$\phi$ is $\spo(2n|2n)$-invariant (i.e. $\C\phi$ is a trivial
module), and thus ${\rm ker}\Delta/\C\phi$ is irreducible.

(iv). First we write down explict formulas for
$$e_0=\xi_n\frac{\partial}{\partial
x_1}-\bar{x}_1\frac{\partial}{\partial \bar{\xi}_n},
 \quad
f_0=x_1\frac{\partial}{\partial
\xi_n}+\bar{\xi}_n\frac{\partial}{\partial \bar{x}_1},
$$
that are the odd simple positive and negative root vectors,
respectively.

We will give a proof only in the most involved case of $k=4$ and
$n\ge 4$. We decompose ${\rm ker}\Delta$ as an
$\mathfrak{sp}(2n)\oplus{\mathfrak so}(2n)$-modules and search among
those highest weights the ones that are finite-dimensional
$\spo(2n|2n)$-highest weights. The only possibilities are
$\sum_{i=1}^4\delta_i$, $\delta_1+\delta_2$, and $0$, each appearing
with multiplicity one. Now it is not difficult to write down the
corresponding $\mathfrak{sp}(2n)\oplus{\mathfrak so}(2n)$-highest
weight vectors of these weights, namely explicitly they are
\begin{align*}
\xi_1\xi_2\xi_3\xi_4,\quad
\xi_1\xi_2(\phi_1-\frac{n}{n-2}\phi_0),\quad
(\phi_1-\phi_0)^2+\frac{1}{n-1}\phi_0^2.
\end{align*}
Next we observe that the central character of $\sum_{i=1}^4\delta_i$
is different from that of $\delta_1+\delta_2$ (which is the same as
that of $0$). Thus if ${\rm ker}\Delta$ were not irreducible, then
among those three vectors at least two would be killed by $e_0$.
Applying $e_0$ to these vectors we find that only
$\xi_1\xi_2\xi_3\xi_4$ is killed by $e_0$. This concludes the proof.
\end{proof}

\begin{remark}
Using a similar argument as the one in the proof of Proposition
\ref{exterior:tensor} (iv) one can show that ${\rm
ker}\left[\Delta:\Lambda^n(\C^{2n|3})\rightarrow\Lambda^{n-2}(\C^{2n|3})\right]$
is irreducible over $\spo(2n|3)$. Also Proposition
\ref{exterior:tensor} (iii) shows that the Jacobi-Trudi character is
in general not irreducible.
\end{remark}

\subsection{A tensor product decomposition}

\begin{proposition} \label{decomp}
Let $\G=\spo(2|2m+1)$, and let $k \ge 2$. Then, the following
$\G$-module decompositions hold:
\begin{eqnarray*}
L_{\delta_1+(k-1)\epsilon_1}\otimes\C^{2|2m+1}
 &\cong&
L_{2\delta_1+(k-1)\epsilon_1} \oplus L_{\delta_1+k\epsilon_1} \oplus
L_{\delta_1+(k-2)\epsilon_1}
 \\
 \C^{2|2m+1} \otimes\C^{2|2m+1}
 &\cong& L_{2\delta_1} \oplus L_{\delta_1+\epsilon_1} \oplus \C.
\end{eqnarray*}
\end{proposition}

\begin{proof}
We will first consider the tensor product
$L_{\delta_1+(k-1)\epsilon_1}\otimes\C^{2|2m+1}$. It follows from
the proof of Proposition~\ref{exterior:tensor}~(i) that, as a
$\G_{\bar{0}} =\mathfrak{sp}(2)\oplus \mathfrak{so}(2m+1)$-module,
$L(\delta_1+(k-1)\epsilon_1)$ (which is identified with the kernel
of $\Delta$ on $\Lambda^k(\C^{2|2m+1})$) is a direct sum of three
irreducible modules with highest weight vectors $v$,
$x_1^k$ and $x_1^{k-1}\xi_1$.  One then check by an elementary but
tedious calculation that $L(\delta_1+(k-1)\epsilon_1)\otimes
\C^{2|2m+1}$ as a $\G_{\bar{0}}$-module is a direct sum of $13$
irreducible modules with respective highest weight vectors given
as follows:
\begin{align*}
&x_1^k\otimes x_1,\quad x_1^{k-1}\left(x_1\otimes x_2-x_2\otimes
x_1\right),
 \\
&x_1^{k-1}\left(x_0\otimes x_0+\sum_{i=1}^m(x_i\otimes
\bar{x}_i+\bar{x}_i\otimes
x_i)\right)-\frac{k-1}{k+m-\frac{3}{2}}x_1^{k-2}(\hf
x_0^2+\sum_{i=1}^m x_i\bar{x}_i)\otimes x_1,
 \\
&x_1^k\otimes\xi,\quad x_1^{k-1}\xi\otimes x_1,\quad
x_1^{k-2}\xi\left(x_1\otimes x_2-x_2\otimes x_2\right),
 \\
&\gamma_0 :=\xi x_1^{k-2}\left(x_0\otimes x_0 +
\sum_{i=1}^m(x_i\otimes\bar{x}_i+\bar{x}_i\otimes x_i) \right) -
 \\
&\qquad\qquad\qquad\qquad\qquad\qquad\frac{k-2}{k+m-\frac{5}{2}}\xi
x_1^{k-3}\left( \hf
x_0^2+\sum_{i=1}^m x_i\bar{x}_i \right)\otimes x_1,
 \\
&x_1^{k-1}\xi\otimes\xi,\quad
x_1^{k-1}(\xi\otimes\bar{\xi}-\bar{\xi}\otimes\xi),\quad \psi_0
x_1^{k-2}\otimes x_1,\quad \psi_0 x_1^{k-3}(x_1\otimes
x_2-x_2\otimes x_1),
 \\
&\psi_0\left(x_1^{k-3}(x_0\otimes x_0 + \sum_{i=1}^m
x_i\otimes\bar{x}_i+\bar{x}_i\otimes x_i)  -
\frac{k-3}{k+m-\frac{7}{2}} x_1^{k-4}( \hf
x_0^2+\sum_{i=1}^m x_i\bar{x}_i)\otimes x_1\right),
 \\
&\psi_0x_1^{k-2}\otimes\xi,
\end{align*}
where
$\psi_0=(k+n-\frac{3}{2})\xi_1\bar{\xi}_1-\sum_{i=1}^mx_i\bar{x}_i-\hf
x_0^2$.

We  observe that among the weights of these vectors only five of
them can possibly be finite-dimensional $\spo(2|2m+1)$-highest
weights, namely $2\delta_1+(k-1)\epsilon_1$, $\delta_1+k\epsilon_1$,
and $\delta_1+(k-2)\epsilon_1$, where the first appears with
multiplicity one and the latter two each appears with multiplicity
two.  Note that these three weights give rise to distinct central
characters and hence the tensor product must be complete reducible,
with each irreducible component generated by a singular vector. Now
both $e_0. \gamma_0$ and $e_0. (x_1^{k-2}\psi_0\otimes\xi)$ are
nonzero and proportional to each other. Thus there is exactly one
$\spo(2|2m+1)$-highest weight vector of highest weight
$\delta_1+(k-2)\epsilon_1$. Furthermore, both $e_0. x_1^k\otimes\xi$
and $e_0. x_1^{k-1}\xi\otimes x_1$ are nonzero and proportional to
each other,
and hence there is exactly one $\spo(2|2m+1)$-highest weight vector
of highest weight $\delta_1+k\epsilon_1$. Clearly
$L(2\delta_1+(k-1)\epsilon_1)$ appears in the tensor product
decomposition with multiplicity one as $2\delta_1+(k-1)\epsilon_1$
is the unique highest weight of multiplicity one. This proves the
first identity.


The decomposition for $\C^{2|2m+1} \otimes\C^{2|2m+1}$ can be
proved in the same way.  Indeed here the situation is simpler,
since as an $\G_{\bar{0}}$-module $\C^{2|2m+1} \otimes\C^{2|2m+1}$
is a direct sum of only eight irreducibles. We skip the details.
\end{proof}

\begin{remark}
Let $\G=\spo(4|5)$.  We can show similarly as for
Proposition~\ref{decomp} that
\begin{equation*}
[L_{\delta_1+\delta_2}\otimes\C^{4|5}] =
\left[L_{2\delta_1+\delta_2} + L_{\delta_1}\right]\oplus
[L_{\delta_1+\delta_2+\epsilon_1}] \oplus [L_{\delta_1}],
\end{equation*}
where $\left[L_{2\delta_1+\delta_2} + L_{\delta_1}\right]$ denotes a
non-trivial extension of modules.
\end{remark}

\section{Examples and a conjecture}
\label{sec:examples}

\subsection{Examples of $\spo(2|3)$}
\label{sec:ex}

Throughout this Section \ref{sec:ex}, we let $\G =\spo(2|3)$. For
$\la=a\delta_1+b\epsilon_1$ write $L(a|b) =L(\la)$, $K(a|b)
=K(\la)$, and $E^\p(a|b)=E^\p(L^0(\la))$. The complete list of
atypical weights are: $\la=a\delta_1+b\epsilon_1$, where $(a|b)
=(\ell|\ell-1)$ for $\ell \ge 1$, or $(0|0)$. If $\la$ is typical,
then the character of $L(\la)$ is equal to $K(\la)$. On the other
hand, the character of $L(\ell|\ell-1)$ was computed in \cite{Ger}
(also see \cite{Gru}). In particular, ${\rm dim}L(\ell|\ell-1)
=2(4\ell^2-1)$ for $\ell \ge 2$, and ${\rm dim}L(1|0) =5.$ All the
examples in this section are computed from the definitions and the
formulas of these irreducible characters. We skip the details.

\subsubsection{``Composition factors" of the Euler
characters}

By abuse of notation, we will regard the Euler characters and the
Kac virtual characters as elements in the Grothendieck group of
the finite-dimensional $\G$-modules. We shall denote by $[L(\la)]$
the element in the Grothendieck group corresponding to the module
$L(\la)$.

\begin{example} \label{ex:euler}
Consider the maximal parabolic subalgebra $\p$ of $\spo(2|3)$
obtained by removing the simple root $\epsilon_1$ so that
$\lev=\gl(1|1)$. If $\la$ is typical, then $E^\p(\la) =K(\la)
=[L(\la)]$. For atypical weights, the following identities of
characters hold:
\begin{align*}
E^\p(\ell+1|\ell) &=K(\ell+1|\ell)=[L(\ell+1|\ell)] +[L(\ell|\ell-1)],\quad \ell \ge 2,\\
E^\p(2|1)&=K(2|1)= [L(2|1)] +[L(1|0)] +[L(0|0)],\\
E^\p(1|0)&=K(1|0)= [L(1|0)] -[L(0|0)],\\
E^\p(0|0)&=2 [L(0|0)].
\end{align*}
\end{example}

\begin{example}
Consider $\spo(2|3)$ with $\p$ obtained by removing the simple
root $\delta_1-\epsilon_1$ so that $\lev\cong
\gl(1)\oplus\mathfrak{so}(3)$. For atypical weights, the following
identities of characters hold:
\begin{align*}
E^\p(\ell+1|\ell) &=K(\ell+1|\ell) = [L(\ell+1|\ell)] +[L(\ell|\ell-1)],\quad \ell\ge 2,\\
E^\p(2|1)&=K(2|1)= [L(2|1)] +[L(1|0)] +[L(0|0)],\\
E^\p(1|0)&=K(1|0)=[L(1|0)] -[L(0|0)],\\
E^\p(0|0)&=K(0|0)=[L(0|0)] -[L(1|0)].
\end{align*}
These identities also follow from \cite[Lemma~2.2.1]{Ger} and our
Remark~2.4.
\end{example}

\subsubsection{Some virtual dimension formulas}

Since each virtual character $D(\la)$ or $K(\la)$ can be written
as the difference of two honest characters,  there is a
well-defined notion of {\em virtual dimension}, $\vdim$, of these
virtual characters as the difference of the degrees of the two
honest characters.

\begin{example}
The following virtual dimension formulas hold:
\begin{enumerate}
\item ${\vdim} K(\la)=
2^{|\Delta_1^+|}\prod_{\alpha\in\Delta_0^+}\frac{(\alpha,\la+\rho)}
{(\alpha,\la+\rho_0)}$.

\item ${\vdim}K(1|0)=4$.

\item ${\vdim}K(0|0)=-4$.

\item ${\vdim} K(\ell|\ell-1) =4(2\ell-1)^2= {\rm
dim}L(\ell|\ell-1) +{\rm dim}L(\ell-1|\ell-2)$, $\ell\ge 3$.

\item ${\vdim}D(\ell|\ell-1)={\vdim} K(\ell|\ell-1) -
(-1)^\ell$, $\ell\ge 3$.

\item ${\vdim}D(2|1) =35={\rm dim}L(2|1) +{\rm dim}L(1|0)$.

\item ${\vdim}D(1|0) ={\rm dim}L(1|0)$; actually, $D(1|0) =L
(1|0)$.

\item ${\vdim}D(1|k-1) ={\rm dim}L(1|k-1)$, for $k \ge 1$.
\end{enumerate}
\end{example}

\subsubsection{Tensor products of the simples with the natural
module}

Below we give explicit formulas for $L(a|b)\otimes L(1|0)$, where
we recall $L(1|0) =\C^{2|3}$, the natural $\spo(2|3)$-module.

\begin{example}
\begin{enumerate}
\item {$\la =(a|b)$ atypical.}
\begin{align*}
&[L(1|0) \otimes L(1|0)] =[L(2|0)] +[L(1|1)] +[L(0|0)],\\
&[L(l|l-1) \otimes L(1|0)] =[L(l+1|l-1)] +[L(l|l)]
+[L(l|l-1)],\quad l\ge 2.
\end{align*}

\item {$\la =(a|b)$ typical with $0\le b\le 1$ or $0\le a \le 1$.}
\begin{align*}
&[L(2|0)\otimes L(1|0)] =[L(3|0)] +[L(2|1)] +2[L(1|0)],\\
&[L(l|0)\otimes L(1|0)] =[L(l+1|0)] +[L(l|1)] +[L(l-1|0)],\quad l\ge 3.\\
&[L(1|1)\otimes L(1|0)] =[L(2|1)] +[L(1|2)] +2[L(1|0)],\\
&[L(3|1)\otimes L(1|0)] =[L(4|1)] +[L(3|1)]  +[L(3|0)] \\
&\qquad\qquad\qquad\qquad\quad  +[L(3|2)] +2[L(2|1)] +2[L(1|0)],\\
&[L(l|1)\otimes L(1|0)] =[L(l|1)] +[L(l+1|1)]\\
&\qquad\qquad\qquad\qquad\quad  +[L(l-1|1)] +[L(l|2)] +[L(l|0)],\quad l\ge 4.\\
&[L(1|l)\otimes L(1|0)] =[L(2|l)] +[L(1|l+1)] +[L(1|l-1)],\quad l\ge
2.
\end{align*}

\item {$\la =(l|l)$ typical with $l\ge 2$.}
\begin{align*}
[L(2|2)\otimes L(1|0)] &=[L(3|2)]+2[L(2|1)] +[L(2|2)] \\
&+[L(1|2)] +[L(2|3)] +[L(1|0)] +[L(0|0)].\\
[L(l|l)\otimes L(1|0)] &=[L(l+1|l)] +2[L(l|l-1)] +[L(l|l)]\\
&+[L(l-1|l)] +[L(l|l+1)] +[L(l-1|l-2)].
\end{align*}

\item {$\la =(l+2|l)$ typical with $l \ge 0$.}
\begin{align*}
[L(l+2|l)\otimes L(1|0)] &= [L(l+3|l)] +2[L(l+1|l)] +[L(l+2|l)] \\
&+[L(l|l-1)] +[L(l+2|l+1)] +[L(l+2|l-1)].
\end{align*}

\item {$\la =(a|b)$ typical and $a,b \ge 2$, $a\not=b+2$,
$a\not=b$.}
\begin{align*}
[L(a|b)\otimes L(1|0)] &= [L(a+1|b)] +[L(a|b)] +[L(a-1|b)] \\
&+[L(a|b+1)] +[L(a|b-1)].\\
\end{align*}
\end{enumerate}
\end{example}

\subsection{A conjecture}

It is known that the parametrization set of highest weights of the
irreducible polynomial representations of $\gl(n|m)$ (Sergeev
\cite{Sv1}) is the same as that of the finite-dimensional simple
$\G=\spo(2n|2m+1)$-modules (see Remark~\ref{rem:hook}).
\begin{conjecture}
Consider the maximal parabolic subalgebra $\p$ of
$\G=\spo(2n|2m+1)$ obtained by removing the simple root
$\epsilon_m$ so that $\lev=\gl(n|m)$. Then
\begin{enumerate}
\item The $E^\p(L^0(\la^\sharp))$, with $\la$ running over all partitions with
$\la_{n+1} \le m$, form a basis for the complexified Grothendieck
group of the category of finite-dimensional $\G$-modules.

\item When $\la$ is ``not close" to the zero weight, then the
composition factors for the Euler character $E^\p(L^0(\la^\sharp))$
is the same as that for the Kac module $K(\la^\sharp)$ for
$\gl(n|m)$.
\end{enumerate}
\end{conjecture}

The comparison of Example~\ref{ex:euler} with some well-known
facts for $\gl(1|1)$ shows that the Conjecture is true for
$\spo(2|3)$.

\end{document}